\documentclass[12pt]{amsart} 
\usepackage[normalem]{ulem}
\usepackage{amsmath,amstext,amsthm,amssymb,amsxtra}

\usepackage{txfonts} 
\usepackage[T1]{fontenc}
\usepackage{lmodern}

\usepackage{euler}  

\usepackage{tikz}

\usepackage{mathtools}
\mathtoolsset{showonlyrefs,showmanualtags}

\usepackage{hyperref} 
\hypersetup{
  colorlinks=true,    
  linkcolor=blue,     
  citecolor=magenta,    
  filecolor=magenta,   
  urlcolor=cyan      
}

\usepackage{amsrefs} 

\allowdisplaybreaks
\swapnumbers
 
\theoremstyle{plain} 
\newtheorem{lemma}[equation]{Lemma} 
\newtheorem{proposition}[equation]{Proposition} 
\newtheorem{theorem}[equation]{Theorem} 
 
\newtheorem{conjecture}[equation]{Conjecture}
\newtheorem{SIZE}[equation]{Size Lemma} 
\newtheorem{dense}[equation]{Density Lemma} 
\newtheorem{tree}[equation]{Tree Lemma} 
\newtheorem{carleson}[equation]{Carleson Theorem}

\theoremstyle{definition}
\newtheorem{definition}[equation]{Definition} 

\theoremstyle{remark}
\newtheorem{remark}[equation]{Remark}

\numberwithin{equation}{section}

\def\norm#1.#2.{\|#1\|_{#2}}
\def\Norm#1.#2.{\bigl\|#1\bigr\|_{#2}}
\def\NOrm#1.#2.{\Bigl\|#1\Bigr\|_{#2}}
\def\NORm#1.#2.{\biggl\|#1\biggr\|_{#2}}
\def\NORM#1.#2.{\Biggl\|#1\Biggr\|_{#2}}

\def\ip#1,#2,{\langle #1,#2\rangle}
\def\Ip#1,#2,{\bigl\langle#1,#2\bigr\rangle}
\def\IP#1,#2,{\Bigl\langle#1,#2\Bigr\rangle}

\def\>{\rangle}
\def\<{\langle}

\def\density{\text{dense}}

\def\P{\mathbf P}
\def\T{\mathbf T}
\def\I{\mathbf I}
\def\Q{\mathbf Q}
\title {On the Convergence of Lacunary Walsh-Fourier Series}
\subjclass[2010]{Primary: 42A20  Secondary:  42B20, 42B25, 42B35}
\author{Yen Do}
\address{ School of Mathematics, Georgia Institute of Technology, Atlanta GA 30332, USA}
\email {yendo@math.gatech.edu}
\thanks{Research supported by a Mathematical Sciences Institute Postdoctoral Fellowship, grant NSF-DMS-0635607002.}

\author{Michael Lacey}  
\address{ School of Mathematics, Georgia Institute of Technology, Atlanta GA 30332, USA}
\email {lacey@math.gatech.edu}
\thanks{Research supported in part by grant NSF-DMS 0968499.}

\begin{document}
\begin{abstract}
We show that for function $ f$  on $ [0,1]$, with $ \int |  f| (\log\log_+ |f| )(\log\log \log_+ |f| ) \; dx < \infty $, and lacunary subsequence of integers $ \{n_j\}$, it holds that $ S _{n_j} f \longrightarrow f $ a.e., where $ S _{m}f$ is the $ m$-th 
Walsh-Fourier partial sum of $ f$. 
According to a result of Konyagin, the sharp integrability condition would not have the triple-log term in it. 
The method of proof uses four ingredients, (1)  analysis on the Walsh Phase Plane, (2) the new multi-frequency Calder\'on-Zygmund Decomposition of Nazarov-Oberlin-Thiele,  (3) a classical inequality of Zygmund, giving an improvement in the Hausdorff-Young  inequality for lacunary subsequences of integers,  and (4) the extrapolation method of Carro-Mart\'in, which generalizes the work of Antonov and Arias-de-Reyna.  
\end{abstract}

\maketitle

\section{Introduction} 

Let $ f $ be an integrable function on the torus $ \mathbb T $, which will be associated with the interval $ [0,1] $ in this paper. 
We consider the Walsh system of functions on $ [0,1]$ given by $ W _{0} (x) =1$, and for $ n \ge 1$, we write 
$ n= \sum_{k=0} ^{r} \varepsilon _k 2 ^{k}$ in binary digits, and define 
\begin{equation*}
W _{n} (x) := \prod _{k=0} ^{r} (\textup{sign} \sin (2 ^{k+1} \pi x)) ^{\varepsilon _{k}} \,. 
\end{equation*}
(We are reserving a lower-case $ w$ for Walsh wave packets defined in \S\ref{s.phasePlane}.)  
The Walsh system is a complete orthonormal system for $ L ^2 (0,1)$, and each $ f\in L ^1$ has the Walsh-Fourier representation 
\begin{equation*}
f \sim \sum_{k\ge 0} \widehat   f (k) W _{k}\,, \qquad \widehat  f (k) := \int _{[0,1]} f (x) W _{k} (x)\; dx  
\end{equation*}
The partial sums of the series above, $ S _{n} f := \sum_{k=0} ^{n } \widehat   f (k) W _{k} $ are the concern of this paper, 
 strongly motivated by the close analogy between the Walsh-Fourier series, and Fourier series.  

The question addressed here, and brought to our attention by the informative article of Konyagin \cite{MR2275651}, is this: 
If a given sequence of integers $ n_j$ is sparse enough, can one assert the pointwise convergence of the Walsh-Fourier sums $ S _{n_j} f$ for a broader class of functions than one has for the full sequence of partial sums?  Indeed, this is the case, as we will show in Theorem~\ref{maintheorem} below; our result comes close to  resolving  Konyagin's  Conjecture~\ref{j.lC} below.  

Let us recall the essential facts.  It is well known that integrability of $ f $ is not enough for this pointwise convergence via a counter example of Kolmogorov \cite{kolmogorov}, who in fact constructed  an integrable function whose Fourier series diverges almost everywhere. Carleson \cite{MR0199631} showed in his seminal paper the almost everywhere convergence of the Fourier series for $ f\in L^2 $, and Hunt \cite{MR0238019} observed extension of Carleson's proof to $ L^p $, $ 1<p<\infty $. These results were reproved by Fefferman \cite{MR0340926}, and the approach of  Lacey-Thiele \cite{MR1783613} is modified in this  paper.  

A result of Konyagin \cite{MR2200228}, extending the work of many, including \cites{MR0096068,MR678901,MR652607}, shows the following. 

\begin{theorem}\label{t.k} Let $ \phi  $ be a increasing convex function such that $ \phi (t) = o (t \log\log t)$ as $ t\to \infty $.  
Then, for any increasing sequence of integers $n_j$ there is a $ f \in \phi (L)$ such that $ \sup _{j} | S _{n_j}f (x)| = \infty $ 
for all $ x\in [0,1]$. 
\end{theorem}

 Konyagin \cite{MR2275651} has conjectured that the previous result is in fact sharp for lacunary subsequences of integers. 
A sequence $ (n_j)_{j\ge 1} $ is called \emph{lacunary} if
$$\inf_{j\ge 1} \frac {n_{j+1}}{n_j} > 1 \ \ .$$
 Below, $\log_+x =  27+\max(0,\log x)$.\footnote{With this definition, you only need to have one $ \log_+$ in the formulas.}

\begin{conjecture}\label{j.lC} Let $ \{n_j \;:\; j\ge 1\} $ be a lacunary sequence of integers. If
$$\int |f(x)| \log \log_+ |f(x)| dx < \infty \ \ ,$$
then $ S_{n_j}f(x) \longrightarrow f(x) $ for almost every $ x \in \mathbb T $. 
\end{conjecture}

The following result about arbitrary lacunary partial sums is formulated in \cite{MR2275651} and is attributed to Zygmund \cite{MR1963498}.
It provides an integrability condition sufficient for convergence of Fourier series better than what is known for the full sequence of integers,  see \S\ref{s.final}. 
\begin{theorem}Let $ \{n_j \;:\; j\ge 1\} $ be a lacunary sequence of integers. If
$$\int_{\mathbb T} |f(x)| \log_+|f(x)| dx < \infty \ \ ,$$
then $ S_{n_j}f(x) $ converges to $ f(x) $ for almost every $ x\in\mathbb T $.
\end{theorem}

Our main results are the following, which do not quite resolve Konyagin's Conjecture, but do certainly indicate that for lacunary sequences, 
one can have convergence for a much broader class of functions than that of the full sequence of partial sums.  
\begin{theorem}\label{maintheorem} 
Let $ \{n_j \;:\; j\ge 1\} $ be a lacunary sequence of integers. If 
$$\int _{\mathbb T} | f| (\log\log_+ |f|)(\log\log \log_+|f| ) \; dx < \infty $$ 
then the Walsh-Fourier partial sums $ S_{n_j}f(x) \longrightarrow f(x) $ for a.e. $ x \in \mathbb T $. Furthermore,
\begin{align}\label{e.LloglogL-weak}
\|\sup_{j} |S_{n_j} f| \|_{1,\infty}& \lesssim  \|f\|_{L(\log\log_+ L) (\log\log\log_{+}L)} \ \ ,
\\
\label{e.LlogL}
\|\sup_{j} |S_{n_j} f| \|_{1} &\lesssim \| f \|_{ L (\log L)(\log\log_{+} L)[0,1]}  \ \ .
\end{align}
\end{theorem}

Note that in \eqref{e.LloglogL-weak} and \eqref{e.LlogL} we use the Luxembourg norms on the right hand sides. We'll recall some standard facts about these norms and  Orlicz spaces at the end of this section. 

Our proof will employ the time-frequency analysis techniques introduced in Lacey-Thiele \cite{MR1783613}, recalled in \S\ref{s.phasePlane},  with an additional key ingredient, namely the multi-frequency Calder\'on-Zygmund decomposition introduced in Nazarov-Oberlin-Thiele \cite{MR2653686}. In a Calder\'on-Zygmund decomposition, one decomposes a function into two parts, where the good part is  bounded and the bad part is localized to a  family of intervals where it has   cancellation properties. The classical Calder\'on-Zygmund decomposition requires the bad part have mean zero on each interval, and in a multi-frequency decomposition one requires certain modulations of the bad part to jointly have mean zero. In \cite{2010arXiv1004.4019O}, Oberlin and Thiele used a Walsh-Paley variant of this decomposition to extend  boundedness results for a Walsh variant of the Bilinear Hilbert transform. In this sense, our proof is a continuation of this theme. In our setting, we are  able to obtain improvement in Carleson's Theorem
when the collection of frequencies is lacunary. Essentially, the estimate of the good part in the multi-frequency decomposition in \cite{MR2653686} is based on Hausdorff-Young's inequality. 
It turns out that if the sequence $ n_1<n_2<\cdots $ is  lacunary, 
one can obtain an improvement of the Hausdorff-Young estimate,  and this is the key to the improvement in Theorem~\ref{maintheorem}. 
The improvement of the Hausdorff-Young estimate is due to Zygmund \cite{MR1963498}, see Proposition~\ref{p.zyg} below. 

 Using the above ingredients, we will be able to show the following refined distributional estimate
$$(\mathscr Cf)^{\ast}(t) \lesssim \frac{|F|}{t} \log_+\log_+ (\frac{t}{|F|}) \ \ , \ \ t> 0 \ \ ,$$
for any $f$ majorized by $F \subset [0,1]$,  see Lemma~\ref{l.FloglogGF}. From this estimate, the strong type estimate \eqref{e.LlogL} follows easily.
To obtain \eqref{e.LloglogL-weak}, we'll use the   extrapolation technique of Antonov~\cite{MR1407066}, which has been extended and generalized in Arias-de-Reyna~\cite{MR1875141} and Carro-Mart\'in \cite{MR2076775}. Details about the proof of Theorem~\ref{maintheorem} are presented in \S\ref{s.strongtype-proof} and \S\ref{s.weaktype-proof}, and  more remarks about this intricate subject are included in \S\ref{s.final}.

\bigskip 
We recall standard facts about Orlicz spaces and Luxembourg norms. 

\begin{definition}\label{d.orlicz} A function $ \psi $ is an \emph{Orlicz} function if it is a convex non-decreasing function on $ [0, \infty ) $ such that 
$$\psi (0)=0 \ \ , \ \ \lim_{t\to \infty } \psi (t) = \infty \ \ .$$
For a probability space $ (\Omega,P) $, we set $ \psi (L) (\Omega ) $ to be those functions $ f $ such that for some $ C>0 $ we have $ \mathbb E \psi (|f| /C) < \infty $, and then define the corresponding Luxembourg norm of $ f $ by
\begin{equation}\label{e.zcL}
\| f\|_{\psi (L)} := \inf \{C \;:\; \mathbb E \psi ( |f| /C) < 1\} \, . 
\end{equation}
If we write $ \psi (L) (I) $ for interval $ I $, we mean that the probability space is $ I $ with normalized Lebesgue measure. 
And by $ \psi (L) $, we mean that the interval is $ [0,1] $.  
\end{definition}

\section{Tiles and Time-Frequency Algorithm} \label{s.phasePlane}

We formulate the details of the Walsh Phase Plane; the linearization of the Carleson operator that is used in the 
rest of the paper; and some key details of the proof of Carleson's Theorem in \cite{MR1783613}.  The Walsh Phase Plane is the closed 
quadrant $ \mathbb R _+ \times \mathbb R _{+}$ of the plane.  A dyadic rectangle is of the form 
\begin{equation}\label{e.tile}
p = I \times \omega = [m 2 ^{j}, (m+1) 2 ^{j}) \times [n 2 ^{k}, (n+1) 2 ^{k}) 
\end{equation}
for integers $ m,n,j,k$.  A \emph{tile} is a dyadic rectangle of area 1, and a \emph{bi-tile} is a dyadic rectangle of area $ 2$.  
For a tile $ I \times \omega $, we will refer to $ I$ as the \emph{time interval associated to $ p$}, and $ \omega $ as the \emph{frequency interval.} 
A bi-tile  $ P$ can be split into a upper half $ P_{\textup{u}} $ and a lower half  $ P_{\ell}$.  Associated to a tile $ p$ is a 
Walsh wave packet $ w _{p}$ which, in the notation or \eqref{e.tile}, is 
\begin{equation}\label{e.wp}
w_p (x) = w _{I \times \omega } (x) := 2 ^{-j/2}W _{n} \bigl( 2 ^{-j} (x - m 2 ^{j})   \bigr) 
\end{equation}
It follows that $ w_p$ has $ L ^2 $ norm one; is supported on $ I$; and is orthogonal to any $ w _{p'}$, where $ p'$  is a second tile 
that does not intersect $ p$.  

The variant of the Carleson operator we consider is defined as follows.  For a  measurable function $ N \;:\; \mathbb R _+ \to \mathbb R _+$, 
we set 
\begin{equation}\label{e.carlesonOp}
\mathscr C f (x) := \sum_{P} \langle f,  w _{P _{\ell}} \rangle   w _{P _{\ell}} (x)\mathbf 1_{  \{ (x, N (x)) \in  P_{\textup{u}} \} }
\end{equation}
The sum is over all bi-tiles $ P\subset \mathbb R _+ \times \mathbb R _+$, and 
we do not indicate the dependence of this definition on the choice of function $ N$.   

And, the Carleson Theorem for the Walsh Phase Plane is 

\begin{carleson}\label{t.Carleson} For $ f \in L ^2 (\mathbb R _+)$ of norm one, and $ G \subset \mathbb R _+ $ of Lebesgue measure one, 
we have the estimate on the bilinear form 
$
\bigl| \langle \mathscr C f , g \rangle\bigr| \lesssim 1 
$, 
where $ 0\le |  g| \le \mathbf 1_{G} $, and the implied constant is absolute.  
\end{carleson}

\begin{remark}\label{r.carlesonOp} 
The Carleson operator \eqref{e.carlesonOp} is the discretization of the maximal operator 
$$\sup_{n\ge 0} |S_n f(x)|$$ 
(c.f. \cite{MR2014553}). In our setting, the supremum is taken over a lacunary subsequence, therefore there will be the following additional restriction that will be in place in subsequent sections:  
With a lacunary sequence $ \{n_j \;:\; j\ge 1\}$ fixed, we can assume that the function $ N (x)$ in \eqref{e.carlesonOp} is defined on (a subset of) $ [0,1]$ and range restricted to $ \{n_j \;:\; j\ge 1\}$.  Thus, we only consider bi-tiles $ P$ so that the upper-half of the frequency interval of $ P$ contains at least one $ n_j$.   Furthermore, we can and will assume that for every bi-tile $P$ in the Carleson operator, the time interval $I_P$ is supported inside $[0,1]$. In particular, this means $\mathscr Cf$ is supported inside $[0,1]$.
\end{remark}

We recall the key elements of the proof of Theorem~\ref{t.Carleson}, following the lines of analysis of \cite{MR1783613}. 
The set of bi-tiles admits a partial order, which we 
write as $ I \times \omega < I' \times \omega '$ iff and only if $ I \subset I'$ and $ \omega ' \subset \omega $.  
It follows that two bi-tiles $ P, P'$ tiles are related by this order if and only if they intersect in the Phase Plane.  We then define 
\begin{equation}\label{e.dense}
\textup{dense} (P) = \sup _{ P' = I' \times \omega ' \;:\; P < P'}  \frac {|  \{ x \in I' \cap G \;:\; (x, N (x)) \in  P' _{ \textup{}}  \}| } {|  I'| } \,. 
\end{equation}
 If $ \mathbf P$ is any collection of tiles, we set $ \textup{dense} (\mathbf P) := \sup _{p\in \mathbf P} \textup{dense} (p)$.  
 
A \emph{tree} is a collection  $ \mathbf T$ of bi-tiles such that there is a (non-unique) bi-tile $I _{\mathbf T} \times \omega _{\mathbf T}$ 
such that $ P <I _{\mathbf T} \times \omega _{\mathbf T} $ for all $ P\in \mathbf T$.  We define 
\begin{equation}\label{e.size}
\textup{size}_f (\mathbf P) 
:= \sup _{\mathbf T }  \Biggl[ |  I_T| ^{-1} \sum_{\substack{P\in \mathbf T\\  P _{\ell }  \cap I _{\mathbf T} \times \omega _{\mathbf T} = \emptyset   }} | \langle f, w _{P _{\ell }} \rangle| ^2   \Biggr] ^{1/2} 
\end{equation}
where the supremum is formed over all trees $\mathbf T \subset \mathbf P$. It is essential to note that the sum is restricted to those tiles $ P \in \mathbf T$ for which the lower-half $ P _{\ell }$ does not intersect the top of the tree.    We add the subscript $ f$ as in the 
application of these concepts, we will be changing $ f$.  

We we will write $ \textup{energy} (\mathbf P) \le A$ if the collection of bi-tiles $ \mathbf P$ is the union of trees $ \mathbf T \in \mathcal T$, 
such that 
\begin{equation*}
\sum_{\mathbf T\in \mathcal T} |  I _{\mathbf T}| \le A \,.  
\end{equation*}

These next Lemmas give a quick proof of Carleson's Theorem, and we will have recourse to them, and their consequences in this paper.  

\begin{dense}
\label{l.dense}
Any collection of tiles $ \mathbf P$ can be written as $ \mathbf P _{\textup{small}} \cup \mathbf P _{\textup{big}}$ 
where these conditions hold. 
\begin{enumerate}
\item   $ \textup{dense} ( \mathbf P _{\textup{small}}) \le \tfrac 12 \textup{dense} ( \mathbf P)$; 
\item   $  \textup{energy} (\mathbf P _{\textup{big}})  \lesssim  \textup{dense} ( \mathbf P) ^{-1} |  G|  $.  
\end{enumerate}
(Recall the role of the set $G $ in Theorem~\ref{t.Carleson} and \eqref{e.dense}.) 
\end{dense}


\begin{SIZE} \label{l.size}
 Any collection of tiles $ \mathbf P$ can be written as $ \mathbf P _{\textup{small}} \cup \mathbf P _{\textup{big}}$ 
where these conditions hold. 
\begin{enumerate}
\item   $ \textup{size}_f ( \mathbf P _{\textup{small}}) \le \tfrac 12 \textup{size}_f ( \mathbf P)$; 
\item   $  \textup{energy} (\mathbf P _{\textup{big}}) \lesssim  \textup{size}_f ( \mathbf P) ^{-2} \| f\|_{2} ^2$.  
\end{enumerate}
(Note the role of $ L ^2 $ in this estimate.)
\end{SIZE}

For collections of tiles $ \mathbf P$ we will use the notation 
\begin{equation}\label{e.BP}
 B _{\mathbf P} (f, g) :=   \sum_{P \in \mathbf P}  \langle f,  w _{P _{\ell}} \rangle  \langle w _{P _{\ell}} \mathbf 1_{ (x, N (x)) \in  P_{\textup{u}}}, g \rangle
\end{equation}

\begin{tree} For any tree $ \mathbf T$ we have the estimate 
\begin{equation}\label{e.tree}
\bigl|  B _{\mathbf T} (f, g) \bigr| \lesssim 
 \textup{dense} ( \mathbf T)  \textup{size}_f ( \mathbf T) |  I _{\mathbf T}| 
\end{equation}
\end{tree}

The next Lemma relates the concept of size to that of the Maximal Function of $ f$. 
It is a consequence of the Calder\'on-Zygmund theory associated with trees.

\begin{lemma}\label{l.upperSize}  Let $ f \in L ^{1} $, $ A > 0$, and let $ \mathbf P$ be a collection of bi-tiles so that for all $ I \times \omega  \in \mathbf P$ we have $ I \cap \{ M f \le  A\} \neq  \emptyset $. We then have 
\begin{equation}\label{e.upperSize}
 \textup{size}_f ( \mathbf P) \lesssim A \,.  
\end{equation}
(In particular, size is bounded by the $ L ^{\infty }$ norm of $ f$.) 
\end{lemma}

To illustrate the Time Frequency Algorithm used in this paper, let us give a proof of Carleson's Theorem, conditional on the Lemmas above.  

\begin{proof}[Proof of Carleson's Theorem]  
We can assume that $ f \in L ^2 \cap L ^{\infty }$. Then, by \eqref{e.upperSize}, it follows that we have an upper bound on the size of the collection of all tiles $ \mathbf P _{\textup{all}}$.  Hence, we have both the Size and Density Lemmas available.  Appropriate inductive application of them leads to a decomposition of $ \mathbf P _{\textup{all}}$ into collections $ \mathbf P _{n}$, for $ n\in \mathbb Z $, such that 
\begin{enumerate}
\item  $ \textup{dense} (\mathbf P_n) \le  \min\{1, 2 ^{-n}\}$; 
\item  $ \textup{size}_f (\mathbf P_n) \le 2 ^{-n/2}$; 
\item  $ \textup{energy} (\mathbf P_n) \lesssim  2 ^{n}$; 
\end{enumerate}
(Note that the density is never more than one, and that  energy estimate matches the conclusions of the Size and Density Lemmas.) 
In particular, $ \mathbf P_n$ is the union of trees $ \mathbf T \in \mathcal T _{n}$ such that we can estimate 
\begin{align*}
 \sum_{P \in \mathbf P_n}\bigl|  \langle f,  w _{P _{\ell}} \rangle  \langle w _{P _{\ell}} \mathbf 1_{ (x, N (x)) \in  P_{\textup{u}}}, g \rangle\bigr|&\lesssim  
 \min\{1, 2 ^{-n}\}  \,  2 ^{-n/2} \sum_{\mathbf T\in \mathcal T_n} |  I _{\mathbf T}| 
\\
& \lesssim  \min\{1, 2 ^{-n}\} \, 2 ^{n/2} = \min \{2 ^{n/2} , 2 ^{-n/2}\} \,, \qquad n\in \mathbb Z \,. 
\end{align*}
The latter estimate is summable over $ n\in \mathbb Z $, so the proof is complete.  
\end{proof}

In subsequent sections, the following situation will appear.  Suppose that a collection $ \mathbf P $ of tiles satisfies
\begin{equation} \label{e.basicAssumption}
   \textup{energy} (\mathbf P)  \lesssim  \textup{dense} (\mathbf P) ^{-1} |  G| \,.  
\end{equation}
It follows that one has the estimate 
\begin{align*} 
\bigl| B _{\mathbf P} (f,g)   \bigr| & \lesssim
  \textup{dense} (\mathbf P)   \textup{size}_f (\mathbf P) 
  \textup{energy} (\mathbf P) 
   \lesssim |  G|   \textup{size}_f (\mathbf P)  \,. 
\end{align*}
We can do better than this estimate if it is more effective to apply the Size Lemma.   
This leads to the following Lemma, also see  \cite{MR1952931}.

\begin{lemma}\label{l.moreEffect} 
Assume that $ \mathbf P$ satisfies \eqref{e.basicAssumption}, and let $ f \in L ^2 $. We have the following estimate 
\begin{equation}\label{e.moreEffect}
\bigl| B _{\mathbf P} (f,g)   \bigr|  
\lesssim 
\min \Bigl\{    \textup{size}_f (\mathbf P)  |  G| \,,  \ 
 \textup{dense} (\mathbf P) ^{1/2}  {\sqrt {|  G| }} {\| f\|_{2}} \Bigr\}  
\end{equation}
\end{lemma}

\begin{proof} The first estimate follows immediately from assumptions and the Tree Lemma.  
So, we assume that the second term on the right in \eqref{e.moreEffect} is the smaller of the two. 
That is, we assume that 
we have the inequality 
\begin{equation*}
 \textup{size} _{f} (\mathbf P) ^{-2} \|  f\| _2 ^2    \le \delta ^{-1} |  G| 
\end{equation*}
where we set $  \textup{dense} (\mathbf P)= \delta $.  
The left hand side of the last display is exactly the estimate on Energy that we would get by application of the Size Lemma.  Hence, it is more 
efficient to apply the Size Lemma until the Energy estimate it provides matches the right hand side of \eqref{e.basicAssumption}. 

To be precise, set integer $ n_0$ to be the integer part of $-\log_2 [\delta \| f\|_{2} ^2 |  G| ^{-1}]   $. 
And write $ \mathbf P$ as the union of collections $ \mathbf P_n$, for $ n\le n_0$, where the collections $ \mathbf P_n$ satisfy 
\begin{enumerate}
\item $  \textup{dense} (\mathbf P_n) \le  \delta$; 
\item $  \textup{size} (\mathbf P _n) \le  2 ^{-n/2}$; 
\item $  \textup{energy} (\mathbf P_n) \lesssim  2 ^{n } \| f\|_{2} ^2 $. 
\end{enumerate}
This decomposition is obtained by solely applying the Size Lemma, until the last step when $n=n_0$, when the conclusions will follow from 
 the assumption \eqref{e.basicAssumption}.  We then have 
\begin{align*}
\bigl| B _{\mathbf P_n} (f,g)\bigr|  & \lesssim \delta  2 ^{n/2}\| f\|_{2}^2 \,,
\end{align*}
which is a geometric series which sums to its at most a constant times it largest term, for $ n=n_0$, yielding our Lemma. 
\end{proof}

\section{A Restricted Weak-Type Inequality} \label{s.rwt}

In this and the subsequent Sections, we shall fix a lacunary sequence $0<n_1<n_2<\dots<n_N$ of frequencies. All the implicit constants in the estimates will be independent of $N$, but depends on the lacunarity constant $\inf_j n_{j+1}/n_j > 1$.

Recall the definition of the Carleson operator $ \mathscr C$ in \eqref{e.carlesonOp}.  The observations of Remark~\ref{r.carlesonOp} will be in force, and we will use the notation for the bilinear form $ B _{\mathbf P}$  in \eqref{e.BP}.  

\begin{definition}\label{d.simple} To say that $ G' $ is a \emph{major subset of a set $ G $ } means that $ G'\subset G $ and $ |G'| \ge \frac 12 |G| $.  
\end{definition}

\begin{lemma}\label{l.FloglogGF} Let  $ F, G \subset [0,1] $. Then there is a major subset $G'$ of $G$ such that if  $ f $ is dominated by $ F $ and $ g $ is dominated by $ G' $ then
\begin{equation}\label{FloglogGF}
\bigl| \langle \mathscr C f, g \rangle\bigr| \lesssim |F| \log \log_+ (|G|/|F|)  \ \ .
\end{equation}
\end{lemma}

We recall the following key inequality of Zygmund \cite{MR1963498},  which can be viewed as an improvement of Hausdorff-Young's inequality in the lacunary setting.

\begin{proposition}\label{p.zyg} Let $ \{n_j \;:\; j\ge 1\} $ be a lacunary sequence of integers. We have the inequality 
\begin{equation*}
\bigl\|\{\widehat f (n_j) \;:\; j\ge 1 \}\bigr\|_{\ell^2 } \lesssim \| f\|_{L (\log L)^{1/2} } \, . 
\end{equation*}
\end{proposition}

We only indicate the proof here.  
The dual space of $ L \sqrt {\log L} $ is the space $ \operatorname {exp}(L ^2 )$, which is the Orlicz space 
associated with the Orlicz function $ \operatorname e ^{x ^2 } -1 $.  Moreover, there is a version of the Khintchine inequality 
which holds for the Walsh-Paley functions $ \{W_{n_j} \;:\; j\ge 1\}$, which is phrased in terms of the $  \operatorname {exp}(L ^2 )$ 
norm.  Namely,
\begin{equation}\label{e.khintchine}
\Bigl\| \sum_{j=1} ^{\infty } a_j W_{n_j} (x)  \Bigr\|_{ \operatorname {exp}(L ^2 )} 
\lesssim \|  \{ a_j \;:\; j\ge 1\} \|_{ \ell ^2 } \,. 
\end{equation}

We can then prove the Zygmund inequality as follows.  For $ f\in L (\log L) ^{1/2}$, let $ \phi = \sum_{j=1} ^{\infty } 
\widehat f (n_j) W_{n_j} $ be the projection of $ f $ onto the lacunary frequencies, and observe that 
\begin{align*}
\bigl\|\{\widehat f (n_j) \;:\; j\ge 1 \}\bigr\|_{\ell^2 }  ^2 
&=  \langle  f, \phi  \rangle 
\\
& \le \| f\|_{ L (\log L) ^{1/2} } \| \phi \|_{ \operatorname {exp}(L ^2 )} 
\\
& \lesssim   \| f\|_{ L (\log L) ^{1/2} }  \cdot 
\bigl\|\{\widehat f (n_j) \;:\; j\ge 1 \}\bigr\|_{\ell^2 }\,. 
\end{align*}
And this completes the proof.  The reader can compare this argument to \cite{MR1963498}*{Ch. XII, 7.6}. 

Concerning \eqref{e.khintchine}, we are sure that this is known, but could not locate an explicit reference to it in the literature. 
One can modify the  argument in \cite{MR1052010} to show the equivalent form of \eqref{e.khintchine} 
\begin{equation*}
\Bigl\| \sum_{j=1} ^{\infty } a_j W_{n_j} (x)  \Bigr\|_p
\le C \sqrt p \,  \|  \{ a_j \;:\; j\ge 1\} \|_{ \ell ^2 } \,, \qquad 1< p < \infty \,.
\end{equation*}
Alternatively, one could show that the Haar Littlewood-Paley Square Function of $ \sum_{j=1} ^{\infty } a_j W_{n_j} $ 
has $ L ^{\infty }$-norm at most $ C \|  \{ a_j \;:\; j\ge 1\} \|_{ \ell ^2 }$, and then appeal to the Chang-Wilson-Wolff  inequality, 
see \cite{MR800004}.

\begin{proof}[Proof of Lemma~\ref{l.FloglogGF}.]

Clearly if $|G|\lesssim |F|$ then the desired estimate follows from $L^2$ boundedness of the (lacunary) Carleson operator. We'll assume below that $|F| \le C_0 |G|$ for some absolute constant $C_0$.

We'll take $G' = G\setminus \{M1_F > \lambda\}$ where we choose $\lambda \simeq |G|^{-1} |F|$ 
so that $G'$ is a major subset of $G$.  Furthermore, we can choose $C_0$ small enough such that $\lambda<1$, and it is then not hard to see that $f$ is supported inside $\{M1_F>\lambda\}$.  Below we show that this choice of $G'$ works.  Let $ \mathbf I$ be the maximal dyadic intervals 
$ I\subset \{M \mathbf 1_{F} > \lambda \}$.  We then have $ \lvert  F \cap I\rvert \lesssim \lambda \lvert  I\rvert  $ for $ I\in \mathbf I$, 
and 
\begin{equation}
 \label{e.sumI}
\sum_{I \in \mathbf I} |I| \le |\{ M 1_F > \lambda\}| \lesssim  \lambda^{-1} |F| \ \ .
\end{equation}

Let $ \mathbf P$ be those bi-tiles $ P$ with $I_{P _{\ell }} \cap F \ne \emptyset$ and $I_{P_{\ell }}  \cap G' \ne \emptyset$. It is clear that
$\< \mathscr Cf, g\> = \mathbf B_\P(f, g)$.  
We then decompose $ \mathbf P = \bigcup_{k \ge 0} \mathbf P_k $,  using only the Density Lemma, see Lemma~\ref{l.dense}. 
Thus, $ \mathbf P_k $ is  a union of trees $ T $ in collection $ \T_k $, so that
$\density(\P_k) \lesssim 2^{-k}$,  and the estimate  (2) of the Density Lemma holds,  namely 
\begin{equation}\label{L1-count}
\textup{energy} (\mathbf P_k)  \le  \sum_{T\in \T_k} \lvert{I_T} \rvert  \lesssim   2^{k} |G|   \,. 
\end{equation}
It follows from \eqref{e.upperSize} and the first half of the estimate \eqref{e.moreEffect}, that we have 
\begin{equation*}
\bigl\lvert B _{\mathbf P_k} (f,g) \bigr\rvert \lesssim \lambda \lvert G\rvert  \lesssim 
\lvert  F\rvert \,, \qquad k\ge 1\,.  
\end{equation*}
We will use this estimate for $ 1 \le k \le  k_0=C \log\log_+ \frac 1{\lambda}$, which is consistent with the estimate we want to prove.   

\medskip 

We begin the multi-frequency part of the proof.  
Fix $k> k_0$. 
For convenience we suppress the dependence on $k$ in the following estimate of $B_{\P_k}(f,g)$, except for $\P_k$.  
For $ I\in \mathbf I$, let $ \mathbf Q_I$ be those tiles $ p$ with time interval $ I$, which as rectangles in the phase plane 
intersect the lower-half of some bi-tile $ P\in \mathbf P_k$.  Take any such pair $(p,P)$. By the construction of $ \mathbf P_k$ and $ \mathbf I$, 
it follows that we must have $ p < P _{\ell }$.  The inequality between $ p$ and the lower part of $ P$ must be strict, hence we must have $ p < P _{u}$.  
 Furthermore, using standard properties of Walsh packets, it follows that $w_{P_{\ell}}$ is a scalar multiple of $w_p$   on $ I _{P _{\ell }}$.  

We set $ \phi _I$ to be the projection of $ f$ onto the space spanned by 
the wave-packets $ \{w_p \;:\; p\in \mathbf Q_I\}$, and set $ \phi = \sum_{I\in \mathbf I} \phi _I$.    This implies that for any $P\in \P_k$ and any $I\in \I$ we have
$$\<f1_I -\phi_I, w_{P_\ell}\> = 0 \ \ .$$
Indeed, since $\phi_I$ is supported on $I$, we can assume $I_{P_\ell} \cap I \ne\emptyset$. Then there will be an element $p \in \Q_I$ such that $p \cap P_\ell \ne\emptyset$, and so we can replace $w_{P_\ell}\mathbf 1_I$ by a multiple $w_p$, and the desired equality follows.

Since $f$ is supported inside the union of intervals in $\I$, it then follows that we have $ B _{\mathbf P_k} (f,g)=B _{\mathbf P_k} (\phi ,g) $, and 
our objective is to use the Zygmund inequality to provide a favorable estimate for the $ L ^2 $-norm of $ \phi $.  

We check that the Zygmund inequality applies to the tiles in $ \mathbf Q_I$. Let $ \alpha >1$ be the lacunarity constant of the 
sequence $ \{n_j \;:\; j\ge 1\}$. For a tile $ p$ in this collection, write 
the frequency interval of $ p$ as $[\mu _p, \mu _p+1] \lvert  I\rvert ^{-1} $. 
Taking a different $ p' \in \mathbf Q _I$, with $ n _{j (p')}> n _{j (p)}   $, we have 
\begin{align*}
\frac { \mu _{p'}} {\mu _{p}} 
\ge  \frac { n _{j (p')} - \lvert  I\rvert ^{-1}   } { n_{j (p)}}  
\ge \alpha - \frac { \lvert  I\rvert ^{-1}   } { n_{j (p)}} \,. 
 \end{align*}
Note that except for the first $1+[ \frac 2{\alpha-1}]$ tiles, we will have $ n _{j (p)} \ge (1+[\frac 2{\alpha-1}])\lvert  I\rvert ^{-1}  > \frac 2{\alpha-1}\lvert  I\rvert ^{-1}$. 
 Hence, after at most $ O _{\alpha } (1)$ initial terms the sequence $ \{ \mu _{p} \;:\; p\in \mathbf Q_I\}$ has lacunarity constant 
at least $ (\alpha +1)/2$, and so the sequence has a lacunarity constant that is only dependent on $ \alpha $.

We estimate as below, where we will be using the 
Zygmund inequality, which requires an appropriate renormalization of the interval $ I$.  
\begin{align*}
\lVert \phi _I\rVert_{2}  & = \lVert  \{ \langle f, w_p \rangle \;:\; p\in \mathbf Q_I\}\rVert_{ \ell ^2 } 
\\
& \lesssim \lVert f \mathbf 1_{I}\rVert_{ L (\log L) ^{1/2} (I)} \lvert  I\rvert ^{1/2}  
\\  
& {\lesssim \frac{|F\cap I|}{|I|} \bigl(\log_+ \frac{|I|}{|F\cap I|} \bigr) ^{1/2} \lvert  I\rvert ^{1/2}}  
\\
& \lesssim \lambda \bigl(\log_+  \frac 1{\lambda} \bigr) ^{1/2} \lvert  I\rvert ^{1/2}   \ \ .
\end{align*}
It follows from \eqref{e.sumI} and $ \lambda \simeq \lvert  F\rvert\cdot \lvert  G\rvert ^{-1}$ that 
\begin{equation*}
\lVert \phi \rVert_{2} \lesssim \lambda \bigl(\log_+  \frac 1{\lambda} \bigr) ^{1/2} \lvert \{M \mathbf 1_{F} > \lambda \}\rvert ^{1/2} \lesssim \lvert  F\rvert   
\cdot \lvert  G\rvert ^{-1/2}  \bigl(\log_+  \frac 1{\lambda}  \bigr) ^{1/2} \,.  
\end{equation*}
    
We now turn to the second half of the estimate \eqref{e.moreEffect} to see that 
\begin{align*} 
\bigl\lvert B _{\mathbf P_k} (f,g)\bigr\rvert 
&= \bigl\lvert B _{\mathbf P_k} (\phi ,g)\bigr\rvert 
\\& 
\lesssim 2 ^{-k/2} \sqrt {\lvert  G\rvert } \cdot  \lVert \phi \rVert_{2} 
\lesssim 2 ^{-k/2} \lvert  F\rvert  \bigl(\log_+  \frac 1{\lambda}  \bigr) ^{1/2} \,.  
\end{align*}
By choice of $ k_0 = C \log\log_+ (\lvert  G\rvert/ \lvert  F\rvert)  $, we can sum this estimate in $ k\ge k_0$ to conclude the proof 
of the Lemma. 
\end{proof}

\begin{remark}\label{r.fromOT} 
Some of the arguments in the proof above we have learned from the last two pages of 
\cite{2010arXiv1004.4019O}, and specialized to the lacunary setting.  
\end{remark}

\section{Proof of the Strong-Type Estimate (\ref{e.LlogL})} \label{s.strongtype-proof}
 
We turn to the proof of the strong-type estimate that $ \mathscr C$ maps $ L \log L (\log\log L)$ into $ L ^{1}$.  
The intermediate inequality we prove is this: 
For any $F\subset [0,1] $ and function $ f $ dominated by $ F $ we have
\begin{align}  \label{e.L1F}
\| \mathscr Cf\|_1 &\lesssim |F| (\log_+ |F|^{-1} ) (\log \log_+|F|^{-1}) \ \ .
\end{align} 

The main idea of the proof is the following principle: one can pass  from a restricted weak-type inequality, together with the $L^2$ estimate, to strong-type inequality with a loss of a $\log$ term. Specifically, by Lemma~\ref{l.FloglogGF}, for any set $ G\subset [0,1] $, there is a major subset $ G'\subset G $ so that for 
any $ g $ dominated by $ G' $, 
\begin{equation} \label{e.iterate}
|\langle \mathscr C f, g\rangle| \lesssim |F|  (\log \log_+|F|^{-1})  \, . 
\end{equation}
We apply \eqref{e.iterate} to the set $ G_0=[0,1] $, getting major subset $ G'_0 $, and then recursively apply the inequality to 
 $ G_1 = G_0 \backslash G'_0 $. After a number $t_0$ of steps, we will have $ | G_{t_0}| \le |F| $, at which point we stop the recursion, and set $ G'_{t_0} = G_{t_0} $. It is not hard to see that we can take
$$ t_0 = 2+ \lceil\log_2 |F|^{-1}\rceil \ \ .$$
Now, for appropriately chosen functions $ g_t $ dominated by $ G'_t $, for $ 0\le t\le t_0 $, we have 
\begin{align*}
 \| \mathscr C f\|_{ 1} & \le 
 \sum_{t=0}^{t_0} \langle \mathscr C f, g_t \rangle 
\\& \lesssim \sum_{t=0}^{t_0} |F|  (\log \log_+|F|^{-1})
\\& \lesssim |F| (\log_+ |F|^{-1} ) (\log\log_+|F|^{-1} ) \, . 
\end{align*}
We use \eqref{e.iterate} for $ 0\le t < t_0 $, and for the last term, we simply use the $ L^2 $ inequality for $ \mathscr C f $. 

\smallskip 
To conclude the inequality \eqref{e.LlogL} from the intermediate \eqref{e.L1F}, one relies upon the fact that an arbitrary 
$ \phi  \in L \log L (\log\log L)$ is a convex combination of functions of the form 
\begin{equation*}
f \cdot  [ \lvert  F\rvert \log_+ \lvert  F\rvert ^{-1}  (\log \log_+ \lvert  F\rvert ^{-1}  )  ] ^{-1}
\end{equation*}
where $ f$ is dominated by $ F$.  We omit the straight-forward proof of this fact.  The reader  should be well-aware that the same comments  do not hold  for the weak-type estimate, which is the focus of the next section.  


\section{A Sj\"olin-type distributional estimate}

Much of the remaining arguments needed to conclude the weak-type estimate \eqref{e.LloglogL-weak} are derived from observations 
brought to bear on the question of the convergence of the Fourier sums along the full sequence.  Here, and in the remainder 
of the paper, we let $ \mathscr C _{\textup{full}} f = \sup _{n\ge 1} \lvert  S_n f\rvert $.  The estimate of Sj\"olin \cite{MR0336222} is 
\begin{equation*}
(\mathscr C _{\textup{full}} \mathbf 1_{F})^{\ast}(t) \lesssim \frac{|F|}{t}  \log_+ (\frac{t}{|F|})\,, \qquad F\subset [0,1]\,,\ t>0\,. 
\end{equation*}
We denote the  decreasing rearrangement invariant function $h^{\ast}$ of function $ h$ by 
\begin{equation}\label{e.**}
h^{\ast}(t) = \inf \{s\ge 0: |\{x:|h(x)|>s\}| \le t\} \ \ .
\end{equation}
Our purpose here is to  establish the lacunary version of this inequality, and then  use an observation of Antonov \cite{MR1407066} to 
obtain a particular extension of this result.

\begin{lemma}\label{l.distribution-restricted} For any $t>0$ and any $f$ majorized by $F \subset [0,1]$, we have
\begin{equation*}
(\mathscr C _{\textup{lac}} {f} )^{\ast}(t) \lesssim \frac{|F|}{t} \log\log_+ \bigl(\frac{t}{|F|}\bigr)\,, \qquad F\subset [0,1]\,,\ t>0\,. 
\end{equation*}
Here and below we set $  \mathscr C _{\textup{lac}} f = \sup _{j\ge 1} \lvert  S_{n_j} f\rvert $.
\end{lemma}

\begin{proof}
For any $s>0$, let $G_s = \{\mathscr Cf(x) > s\}$. By Lemma~\ref{l.FloglogGF}, there exists a major set $G'_s$ such that for
appropriate $ g$  majorized by $G'_s$ we have
\begin{align*}
s|G_s| \le 2 s \lvert  G'_s\rvert & \lesssim \langle \mathscr C f, g \rangle 
\lesssim \lvert  F\rvert  \log\log_+ (|G_s|/|F|)  
\end{align*}
or equivalently
$s \lesssim (|F|/|G_s|) \log \log_+ (|G_s|/|F|)$.  
Therefore if $s \ge C\frac{|F|}{t} \log \log_+ (\frac{t}{|F|})$ for some large absolute constant $C$ we'll have
$$\frac{|F|}{t} \log \log_+ \frac{t}{|F|} \le \frac{|F|}{|G_s|} \log \log_+  \frac{|G_s|}{|F|} \ \ .$$
So by the strictly increasing property of $s\log \log_+(1/s)$, we obtain
$$|F|/t \le |F|/|G_s| \ \ ,$$
or equivalently
$|G_s| \le t$. 
This completes the proof of the Lemma.
\end{proof}

Now, we will use a key observation of Antonov \cite{MR1407066} to remove the restricted-type assumption in Lemma~\ref{l.distribution-restricted}. Unlike previous Lemmas where the bi-tiles in the definition of $\mathscr C$ could be arbitrary, in the following Lemma (and hence subsequent Lemmas) we need to know that $\mathscr C$ is the actual discretization of the lacunary maximal Carleson operator used in the Lemma above. 
The observation of Antonov is the following Lemma.

\begin{lemma}\label{l.antonov} For $M \in \mathbb N$, set $ S ^{M} f = \sup _{1\le n \le M} |S_n f| $. 
For every $\epsilon >0$, and function $0\le f \le 1 $ supported in $[0,1]$, there is a set $ F \subset [0,1]$ with $ \| f\|_{1}= |F| $ and moreover 
$ \| S^{M} (f - \mathbf 1_{F})\|_{\infty } < \epsilon $. 
\end{lemma}
A proof of Antonov's lemma in the Walsh-Fourier setting could be found in Sj\"olin-Soria \cite{MR2014553}.

Using Antonov's observation, have there holds the  following. For  any $t>0$, we have
\begin{align}\label{e.distribution-L1}
(\mathscr C_{\textup{lac}}f)^{\ast}(t) \lesssim \frac{\|f\|_1}{t} \log\log_+ \bigl(\frac{t}{\|f\|_1}\bigr)\,, \qquad  0\le \lvert  f\rvert \le 1\,.  
\\  \label{e.SJO}
(\mathscr C_{\textup{full}}f)^{\ast}(t) \lesssim \frac{\|f\|_1}{t} \log_+ \bigl(\frac{t}{\|f\|_1}\bigr)\,, \qquad  0\le \lvert  f\rvert \le 1\,.  
\end{align}

\section{Proof of the Weak-Type Estimate (\ref{e.LloglogL-weak})}\label{s.weaktype-proof} 

Antonov~\cite{MR1407066} used \eqref{e.SJO} to derive his conclusion that $ \mathscr C _{\textup{full}}$ maps 
$ L \log L \log\log\log_{+}L$ into $ L ^{1,\infty }$, which remains the best known result for the full sequence of integers.  
His argument was further generalized by Arias-de-Reyna \cite{MR1875141}, which language was phrased in that of 
interpolation and extrapolation theory.  The latter approach has been revisited by others, with the relevant point for us 
that the starting point is the distributional inequality \eqref{e.SJO}, or more generally something of the broad form of 
\eqref{e.SJO} or \eqref{e.distribution-L1}.  In our setting, we are fortunate that the investigations of  Carro-Martin \cite{MR2076775}
are nicely suited to derive our weak-type estimate.  

We will be a little brief about this, as the Theorems of Carro-Martin apply in an uncomplicated fashion.  
The extrapolation theory of Carro--Mart\'in starts with  a sublinear operator $T$ such that for any $f \in L^1 \cap L^\infty [0,1]$ with $\|f\|_{\infty} \le 1$ we have, using the notation of \eqref{e.**}, 
\begin{align}\label{e.factorization-Carro-Martin}
(Tf)^{\ast}(t) \lesssim D(\|f\|_1) R(t) \ \ .
\end{align}
Then under mild assumptions on $D$ and $R$, Carro and Martin shows that $T$  is bounded from a logarithmic type space $Q_D$ to a weighted Lorentz space $M_R$. 

For convenience we shall refer to those functions $f \in L^1 \cap L^\infty [0,1]$ with $\|f\|_{\infty} \le 1$ as atoms. 

\begin{definition}\label{d.QD} Let $D:(0,\infty)\to (0,\infty)$ be a concave function such that $D(0+)=0$. Then $Q_D$ is the space of functions $f$ such that there exists a decomposition of $f$ 
$$f=\sum_k a_k f_k \ \ , \ \ a_k \ge 0 \ \ ,$$
where $ \{f_k\}$ are atoms,   
and a scalar partition of unity $\sum_k b_k =1$ (with $b_k \ge 0$) such that the following sum is finite:
$$\sum_k a_k D(\|f_k\|_1) \bigl(1+\log \frac 1 {b_k}\bigr) < \infty \ \ .$$
The infimum of all such sums is denoted by $\|f\|_{Q_D}$.
\end{definition}

\begin{definition}\label{d.MR} Let $R:(0,\infty) \to (0,\infty)$. Then $M_R$ is the space of functions $f$ such that
$$\|f\|_{M_R}:= \sup_{t>0} \frac{f^{\ast}(t)}{R(t)} < \infty \ \ .$$
In particular, when $R(t)=1/t$ the space $M_R$ becomes the usual $L^{1,\infty}$. 
\end{definition}

In the following theorems, extracted from \cite{MR2076775}, we assume that $D, R$ are respectively eligible for the above definitions.

\begin{theorem}\cite{MR2076775}*{Theorem 2.1} Assume that $T$ is a sublinear operator such that for any atomic $f$ and any $t > 0$ we have
$$(Tf)^{\ast}(t) \lesssim D(\|f\|_1) R(t) \ \ .$$
Assume that $tR(t)$ is a nondecreasing function. Then $T$ is bounded from $Q_D$ to $M_R$.
\end{theorem}

Next, the space $ Q_D$ and Orlicz spaces are related.  

\begin{definition}\label{d.LlogloglogLD} Let $D:(0,\infty)\to (0,\infty)$ be a concave increasing function such that $D(0+)=0$. Then $L\log\log\log_{+}L(D)$ is the space of functions $f$ such that
$$\|f\|_{L \log\log\log_{+}L (D)}:= \int_0^1 f^{\ast}(t) \log\log\log_+\frac 1t \;  dD(t) < \infty \ \ .$$
\end{definition}

\begin{theorem}\cite{MR2076775}*{Theorem 2.2(3)} Assume $D$ is also increasing. If $D(s) \gtrsim s$ for any $s>0$, and $D(s^2) \lesssim sD(s)$ for any $0 < s \le 1$, then
$$L \log \log \log L (D) \subset Q_D \ \ .$$
\end{theorem}

\subsection{Proof of the weak type estimate (\ref{e.LloglogL-weak})}

In our case, we'll have \eqref{e.factorization-Carro-Martin} with
$$D(s) = s \log \log_+ \bigl(\frac 1 s\bigr)\,, \qquad  R(t) = \frac 1 t \ \ .$$
To see this, note that this follows from \eqref{e.distribution-L1} if $t\le 1$. When $t>1$, using the trivial bound
$$|\{\mathscr Cf(x)>s\}| \le 1$$ 
which holds for any $s$ (since $\mathscr Cf$ is supported in $[0,1]$), we obtain $(\mathscr Cf)^{\ast}(t) = 0$, and the factorization estimate \eqref{e.factorization-Carro-Martin} follows immediately. We note that technically the above function $D(s)$ is not concave for some range of $s$ near $1$, so what really happens is we use a concave approximation of $D$ that is comparable (up to the first derivative) to $D$  near these values of $s$. We will abuse notation and use $D$ in the sequel without any further comment.

Now, the extrapolation method of Carro-Mart\'in will give the following estimate:
\begin{align}\label{l.post-Carro-Martin}
\|\mathscr C f\|_{1,\infty} &\lesssim \int_0^1 f^{\ast}(t) (\log\log\log_+ \frac 1 t) D'(t) \; dt
\\ &\lesssim \int_0^1 f^{\ast}(t) \bigl(\log\log_+ \frac 1 t\bigr)\bigl(\log\log\log_+ \frac 1 t\bigr)  \; dt
\\ \nonumber & \simeq  \|f\|_{L(\log\log L) (\log \log \log L)} \ \ ,
\end{align}
note that the last equivalence is a known equivalent way to express the Orlicz norm. As it is classical that we have $S_{n_j}f\to f$ a.e. for bounded functions, which are dense in $L (\log\log L) (\log\log \log L)$, this proves the a.e. convergence for all $f$ in this space.

\section{Concluding Remarks} \label{s.final}

For a lacunary sequence of integers $ \{n_j\}$, there is a direct way to see that $ S _{n_j} f$ converges to $ f$ a.e. for $ f\in L (\log L) ^{1/2} $. We indicate this here.  Letting $ V _{n}f$ denote the de la Vall\'ee Poussin sums, we of course have $ V _{n}f$ converging a.e. to $ f$. 
And, one can see that the inequality below 
\begin{equation*}
\Bigl\|  \Bigl[\sum_{j} |  V_{n_j} f - S_{n_j}f | ^2   \Bigr] ^{1/2}  \Bigr\|_{1, \infty } \lesssim 
\| f\|_{L (\log L) ^{1/2} }
\end{equation*}
is a corollary to the endpoint Marcinkiewicz multiplier theorem of Tao-Wright \cite{MR1900894}.  This paper has interesting variants of the Zygmund inequality.  

If we consider the full sequence of partial Walsh-Fourier sums, we have no better estimate than Hausdorff-Young to use in the multi-frequency argument.  We have not seen an estimate that would improve our knowledge of the convergence of the full sequence of Walsh-Fourier sums.  
 Indeed, if we consider any sequence that grows more slowly than lacunary, it would seem that only the Hausdorff-Young inequality is available in the multi-frequency argument.  
  
Konyagin has  showed that for special sequences of indices the corresponding partial sums of the Walsh-Paley series converge almost everywhere to $ f $ for \emph{any}  $ f\in L^1(\mathbb T) $ \cite{MR1256607}. These are sequences of indices such that if we write each index in binary form then there there is an uniform bound on the number of times the digits alternate between $ 0$ and $ 1$. In particular, the sequence of powers of $ 2 $ falls into this category, although it is not hard to construct a lacunary sequence of integers without this property. He has 
  posed the question of characterizing those sequence of integers $ \{n_j\}$ for which the Walsh-Fourier series $ S _{n_j} f$ converge pointwise to $ f$ for all integrable $ f$, see \cite{MR2275651}*{Problem 3.3}.  There are more points of interest in this paper; the interested reader is encouraged to read it.



\begin{bibdiv}
\begin{biblist}
\bib{MR1407066}{article}{
  author={Antonov, N. Yu.},
  title={Convergence of Fourier series},
  booktitle={Proceedings of the XX Workshop on Function Theory (Moscow, 1995)},
  journal={East J. Approx.},
  volume={2},
  date={1996},
  number={2},
  pages={187--196},
  issn={1310-6236},
  review={\MR {1407066 (97h:42005)}},
}

\bib{MR1875141}{article}{
  author={Arias-de-Reyna, J.},
  title={Pointwise convergence of Fourier series},
  journal={J. London Math. Soc. (2)},
  volume={65},
  date={2002},
  number={1},
  pages={139--153},
  issn={0024-6107},
  review={\MR {1875141 (2002k:42009)}},
  doi={10.1112/S0024610701002824},
}

\bib{MR0199631}{article}{
  author={Carleson, Lennart},
  title={On convergence and growth of partial sums of Fourier series},
  journal={Acta Math.},
  volume={116},
  date={1966},
  pages={135--157},
  issn={0001-5962},
  review={\MR {0199631 (33 \#7774)}},
}

\bib{MR2076775}{article}{
  author={Carro, Mar{\'{\i }}a J.},
  author={Mart{\'{\i }}n, Joaquim},
  title={Endpoint estimates from restricted rearrangement inequalities},
  journal={Rev. Mat. Iberoamericana},
  volume={20},
  date={2004},
  number={1},
  pages={131--150},
  issn={0213-2230},
  review={\MR {2076775 (2005d:46153)}},
}

\bib{MR800004}{article}{
  author={Chang, S.-Y. A.},
  author={Wilson, J. M.},
  author={Wolff, T. H.},
  title={Some weighted norm inequalities concerning the Schr\"odinger operators},
  journal={Comment. Math. Helv.},
  volume={60},
  date={1985},
  number={2},
  pages={217--246},
  issn={0010-2571},
  review={\MR {800004 (87d:42027)}},
  doi={10.1007/BF02567411},
}

\bib{MR0340926}{article}{
  author={Fefferman, Charles},
  title={Pointwise convergence of Fourier series},
  journal={Ann. of Math. (2)},
  volume={98},
  date={1973},
  pages={551--571},
  issn={0003-486X},
  review={\MR {0340926 (49 \#5676)}},
}

\bib{MR0096068}{article}{
  author={Gosselin, Richard P.},
  title={On the divergence of Fourier series},
  journal={Proc. Amer. Math. Soc.},
  volume={9},
  date={1958},
  pages={278--282},
  issn={0002-9939},
  review={\MR {0096068 (20 \#2565)}},
}

\bib{MR0238019}{article}{
  author={Hunt, Richard A.},
  title={On the convergence of Fourier series},
  conference={ title={Orthogonal Expansions and their Continuous Analogues (Proc. Conf., Edwardsville, Ill., 1967)}, },
  book={ publisher={Southern Illinois Univ. Press}, place={Carbondale, Ill.}, },
  date={1968},
  pages={235--255},
  review={\MR {0238019 (38 \#6296)}},
}

\bib{kolmogorov}{article}{
  author={Kolmogorov, A},
  title={Une s\'erie de Fourier–Lebesgue divergente presque partout},
  journal={Fund. Math.},
  volume={4},
  date={1923},
  pages={324-328},
}

\bib{MR1256607}{article}{
  author={Konyagin, S. V.},
  title={On a subsequence of Fourier-Walsh partial sums},
  language={Russian, with Russian summary},
  journal={Mat. Zametki},
  volume={54},
  date={1993},
  number={4},
  pages={69--75, 158},
  issn={0025-567X},
  translation={ journal={Math. Notes}, volume={54}, date={1993}, number={3-4}, pages={1026--1030 (1994)}, issn={0001-4346}, },
  review={\MR {1256607 (95e:42030)}},
  doi={10.1007/BF01210421},
}

\bib{MR2200228}{article}{
  author={Konyagin, S. V.},
  title={Divergence everywhere of subsequences of partial sums of trigonometric Fourier series},
  journal={Proc. Steklov Inst. Math.},
  date={2005},
  number={Function Theory, suppl. 2},
  pages={S167--S175},
  issn={0081-5438},
  review={\MR {2200228 (2006j:42007)}},
}

\bib{MR2275651}{article}{
  author={Konyagin, Sergey V.},
  title={Almost everywhere convergence and divergence of Fourier series},
  conference={ title={International Congress of Mathematicians. Vol. II}, },
  book={ publisher={Eur. Math. Soc., Z\"urich}, },
  date={2006},
  pages={1393--1403},
  review={\MR {2275651 (2008b:42006)}},
}

\bib{MR652607}{article}{
  author={K{\"o}rner, T. W.},
  title={Everywhere divergent Fourier series},
  journal={Colloq. Math.},
  volume={45},
  date={1981},
  number={1},
  pages={103--118 (1982)},
  issn={0010-1354},
  review={\MR {652607 (83h:42010)}},
}

\bib{MR1783613}{article}{
  author={Lacey, Michael T.},
  author={Thiele, Christoph},
  title={A proof of boundedness of the Carleson operator},
  journal={Math. Res. Lett.},
  volume={7},
  date={2000},
  number={4},
  pages={361\ndash 370},
  issn={1073-2780},
}

\bib{MR1952931}{article}{
  author={Muscalu, C.},
  author={Tao, T.},
  author={Thiele, C.},
  title={A discrete model for the bi-Carleson operator},
  journal={Geom. Funct. Anal.},
  volume={12},
  date={2002},
  number={6},
  pages={1324--1364},
  issn={1016-443X},
  review={\MR {1952931 (2004b:42043)}},
  doi={10.1007/s00039-002-1324-0},
}

\bib{MR2653686}{article}{
  author={Nazarov, Fedor},
  author={Oberlin, Richard},
  author={Thiele, Christoph},
  title={A Calder\'on-Zygmund decomposition for multiple frequencies and an application to an extension of a lemma of Bourgain},
  journal={Math. Res. Lett.},
  volume={17},
  date={2010},
  number={3},
  pages={529--545},
  issn={1073-2780},
  review={\MR {2653686}},
}

\bib{2010arXiv1004.4019O}{article}{
  author={Oberlin, Richard},
  author={Thiele, Christoph},
  title={New uniform bounds for a Walsh model of the bilinear Hilbert transform},
  journal={ArXiv e-prints},
  journal={IUMJ},
  year={2010},
  pages={to appear},
}

\bib{MR1052010}{article}{
  author={Sagher, Yoram},
  author={Zhou, Ke Cheng},
  title={Local norm inequalities for lacunary series},
  journal={Indiana Univ. Math. J.},
  volume={39},
  date={1990},
  number={1},
  pages={45--60},
  issn={0022-2518},
  review={\MR {1052010 (91b:42017)}},
  doi={10.1512/iumj.1990.39.39005},
}

\bib{MR0336222}{article}{
  author={Sj{\"o}lin, Per},
  title={Convergence almost everywhere of certain singular integrals and multiple Fourier series},
  journal={Ark. Mat.},
  volume={9},
  date={1971},
  pages={65--90},
  issn={0004-2080},
  review={\MR {0336222 (49 \#998)}},
}

\bib{MR2014553}{article}{
  author={Sj{\"o}lin, Per},
  author={Soria, Fernando},
  title={Remarks on a theorem by N. Yu.\ Antonov},
  journal={Studia Math.},
  volume={158},
  date={2003},
  number={1},
  pages={79--97},
  issn={0039-3223},
  review={\MR {2014553 (2004i:42006)}},
  doi={10.4064/sm158-1-7},
}

\bib{MR1900894}{article}{
  author={Tao, Terence},
  author={Wright, James},
  title={Endpoint multiplier theorems of Marcinkiewicz type},
  journal={Rev. Mat. Iberoamericana},
  volume={17},
  date={2001},
  number={3},
  pages={521--558},
  issn={0213-2230},
  review={\MR {1900894 (2003e:42014)}},
}

\bib{MR678901}{article}{
  author={Totik, V.},
  title={On the divergence of Fourier series},
  journal={Publ. Math. Debrecen},
  volume={29},
  date={1982},
  number={3-4},
  pages={251--264},
  issn={0033-3883},
  review={\MR {678901 (84c:42008)}},
}

\bib{MR1963498}{book}{
  author={Zygmund, A.},
  title={Trigonometric series. Vol. I, II},
  series={Cambridge Mathematical Library},
  edition={3},
  note={With a foreword by Robert A. Fefferman},
  publisher={Cambridge University Press},
  place={Cambridge},
  date={2002},
  pages={xii; Vol. I: xiv+383 pp.; Vol. II: viii+364},
  isbn={0-521-89053-5},
  review={\MR {1963498 (2004h:01041)}},
}

\end{biblist}
\end{bibdiv}
\end{document}